\documentclass[reqno,11pt]{amsart}
\usepackage{amssymb}
\usepackage{amsmath}
\usepackage{stmaryrd}
\usepackage{graphicx}
\usepackage{color}
\usepackage{cite}
\usepackage[font=small]{caption}

\usepackage{tikz}
\usepackage{pgfplots}
\pgfplotsset{compat=1.10}
\usepgfplotslibrary{fillbetween}
\usetikzlibrary{patterns}
\newtheorem{theorem}{Theorem}[section]
\newtheorem{corollary}[theorem]{Corollary}
\newtheorem{lemma}[theorem]{Lemma}
\newtheorem{proposition}[theorem]{Proposition}

\newtheorem{remark}[theorem]{Remark}

\newcommand{\R}{\mathbb{R}}
\newcommand{\N}{\mathbb{N}}
\newcommand{\rd}{\mathrm{d}}

\definecolor{cadmiumgreen}{rgb}{0.0, 0.42, 0.24}
\numberwithin{equation}{section}
\numberwithin{figure}{section}
\begin{document}

\title[ ]{Reinforced Limit of a MEMS Model with Heterogeneous Dielectric Properties}

\author{Philippe Lauren\c{c}ot}
\address{Institut de Math\'ematiques de Toulouse, UMR~5219, Universit\'e de Toulouse, CNRS \\ F--31062 Toulouse Cedex 9, France}
\email{laurenco@math.univ-toulouse.fr}
\author{Katerina Nik}
\address{Leibniz Universit\"at Hannover\\ Institut f\" ur Angewandte Mathematik \\ Welfengarten 1 \\ D--30167 Hannover\\ Germany}
\email{nik@ifam.uni-hannover.de}
\author{Christoph Walker}
\address{Leibniz Universit\"at Hannover\\ Institut f\" ur Angewandte Mathematik \\ Welfengarten 1 \\ D--30167 Hannover\\ Germany}
\email{walker@ifam.uni-hannover.de}
%
\thanks{Partially supported by the CNRS Projet International de Coop\'eration Scientifique PICS07710}
\date{\today}
%
%
\begin{abstract}
A MEMS model with an insulating layer is considered and its reinforced limit is derived by means of a Gamma convergence approach when the thickness of the layer tends to zero. The limiting model inherits the dielectric properties of the insulating layer.
\end{abstract}
%
\maketitle
%
\section{Introduction}

Idealized microelectromechanical systems (MEMS)  consist of two dielectric plates: a rigid ground plate above which an elastic plate is suspended. The latter is electrostatically actuated by a Coulomb force which is induced across the device by holding the two plates at different voltages. In this set-up there is thus a competition between attractive electrostatic forces and restoring mechanical forces due to the elasticity of the plate. When the two plates are not prevented from touching each other, a contact of the plates commonly leads to an instability of the device --~also known in the literature as ``pull-in instability''~-- which is revealed as a singularity in the corresponding mathematical equations, e.g., see \cite{LWBible,PeB03} and the references therein. In contrast, when the ground plate is coated with an insulating layer preventing a direct contact of the plates, see Figure~\ref{F1}, a touchdown of the elastic plate on this layer does not result in an instability as the device may continue to operate without interruption (though it still leads to a peculiar situation from a mathematical point of view).  Different mathematical models describing this setting including an insulating layer were introduced \cite{AmEtal,BG01,LW19,YZZ12,LLG14,LLG15}. The basic assumption in all these models is that the state of the device is fully described by the vertical deflection of the elastic plate and the electrostatic potential in the device. According to \cite{BG01,LW19,LLG14,LLG15} the dynamics of the former is governed by an evolution equation while that of the latter is governed by an elliptic equation in a time-varying domain enclosed by the two plates. Due to the heterogeneity of the dielectric properties of the device, this elliptic equation is actually a transmission problem 
 (see \eqref{TMP} below) on the non-smooth time-dependent domain with a transmission condition at the interface separating the insulating layer and the free space. The analysis of such a model turns out to be quite involved \cite[Section~5]{LW19}.  Therefore, several simpler and more tractable models were derived on the assumption of a vanishing aspect ratio of the device \cite{AmEtal,BG01,LW19,LLG14,LLG15}. Thanks to this approximation the electrostatic potential can be computed explicitly in terms of the deflection of the elastic plate, and the model thus reduces to a single equation for the deflection.

The aim of the present work is to derive an intermediate model by letting only the thickness of the insulating  layer go to zero (instead of the aspect ratio of the device). Our starting point is the model analyzed in \cite{LW19} in which we introduce an appropriate scaling of the dielectric permittivity in dependence on the layer's thickness (see \eqref{sigmad} below) and use a Gamma convergence approach to study the limiting behavior. The specific choice of the scaling is required in order to keep relevant information of the dielectric heterogeneity of the device and can be interpreted as a reinforced limit from a mathematical point of view \cite{AB86}.

To be more precise, we recall the model stated in \cite[Section 5]{LW19}. Let $D\subset\R^n$ with $n\ge 1$ be a bounded $C^2$-domain representing the (identical) horizontal cross-section of the two plates (actually, only the cases $n\in\{1,2\}$ are physically relevant for applications to MEMS, the ground plate being $D\times (-H-d,-H)$ and thus a two or three dimensional object). The  dielectric layer of thickness  $\delta>0$ on top of the ground plate located at $z=-H-\delta$ with $H>0$ is then given by
$$
\mathcal{R}_\delta  :=  D\times (-H-\delta,-H)\,.
$$
The deflection of the elastic plate from its rest position at $z=0$ is described by a function  $u:\bar{D}\rightarrow [-H, \infty)$ with $u=0$ on $\partial D$ so that
$$
 \Omega(u):=\left\{(x,z)\in D\times \mathbb{R}\,:\, -H<  z <  u(x)\right\}
$$
is the free space between the elastic plate and the top of the dielectric layer.
We let 
$$
\Sigma (u):=\{(x,-H)\,:\, x\in D,\, u(x)>-H\}
$$
denote the interface separating free space and dielectric layer and put
$$
 \Omega_\delta({u}):=\left\{(x,z)\in D\times \mathbb{R} \,:\, -H-\delta<  z <  u(x)\right\}=\mathcal{R}_\delta \cup  \Omega( {u})\cup  \Sigma(u)\,.
$$
If the elastic plate and the insulating layer remain separate, that is, if $u>-H$ in $D$, then $\Sigma(u)$ coincides with
$$
\Sigma :=D\times \{-H\}\,.
$$
In contrast, a touchdown of the elastic plate on the insulating layer corresponds to a non-empty {\it coincidence set}
$$
\mathcal{C}(u):=\{x\in D\, :\, u(x)=-H\}
$$
and a different geometry as the free space $\Omega(u)$ then has several connected components. It is worth pointing out that in this case --~independent of the smoothness of the function $u$~-- these components may not be Lipschitz domains, a feature which requires some special care in the mathematical analysis.

The different situations with empty and non-empty coincidence sets are depicted in Figure~\ref{F1}.
 \begin{figure}
 	\begin{tikzpicture}[scale=0.9]
 	\draw[black, line width = 1.5pt, dashed] (-7,0)--(7,0);
 	\draw[black, line width = 2pt] (-7,0)--(-7,-5);
 	\draw[black, line width = 2pt] (7,-5)--(7,0);
 	\draw[black, line width = 2pt] (-7,-5)--(7,-5);
 	\draw[black, line width = 2pt] (-7,-4)--(7,-4);
 	\draw[black, line width = 2pt, fill=gray, pattern = north east lines, fill opacity = 0.5] (-7,-4)--(-7,-5)--(7,-5)--(7,-4);
 	\draw[cadmiumgreen, line width = 2pt] plot[domain=-7:7] (\x,{-1-cos((pi*\x/7) r)});
 	\draw[blue, line width = 2pt] plot[domain=-7:-3] (\x,{-2-2*cos((pi*(\x+3)/4) r)});
 	\draw[blue, line width = 2pt] (-3,-4)--(1,-4);
 	\draw[blue, line width = 2pt] plot[domain=1:7] (\x,{-2-2*cos((pi*(\x-1)/6) r)});
 	\draw[cadmiumgreen, line width = 1pt, arrows=->] (3,0)--(3,-1.15);
 	\node at (3.2,-0.6) {${\color{cadmiumgreen} v}$};
 	\draw[blue, line width = 1pt, arrows=->] (-5,0)--(-5,-1.85);
 	\node at (-4.8,-1) {${\color{blue} w}$};
 	\node[draw,rectangle,white,fill=white, rounded corners=5pt] at (2,-4.5) {$\Omega_1$};
 	\node at (2,-4.5) {$\mathcal{R}_\delta $};
 	\node at (-2,-3) {${\color{cadmiumgreen} \Omega(v)}$};
 	\node at (3.75,-5.75) {$D$};
	\draw (3.55,-5.75) edge[->,bend left, line width = 1pt] (2.3,-5.1);
 	\node at (7.7,-3) {$\Sigma(w)$};
 	\draw (7.2,-3) edge[->,bend right, line width = 1pt] (5.2,-3.9);
 	\node at (-7.8,1) {$z$};
 	\draw[black, line width = 1pt, arrows = ->] (-7.5,-6)--(-7.5,1);
 	\node at (-8.4,-5) {$-H-\delta$};
 	\draw[black, line width = 1pt] (-7.6,-5)--(-7.4,-5);
 	\node at (-8,-4) {$-H$};
 	\draw[black, line width = 1pt] (-7.6,-4)--(-7.4,-4);
 	\node at (-7.8,0) {$0$};
 	\draw[black, line width = 1pt] (-7.6,0)--(-7.4,0);
 	\node at (1,-3) {${\color{blue} \mathcal{C}(w)}$};
 	\draw (0.45,-3) edge[->,bend right,blue, line width = 1pt] (-0.5,-3.95);
 	\end{tikzpicture}
 	\caption{Geometry of $\Omega_\delta(u)$ when $n=1$ for a state $u=v$ with empty coincidence set (\textcolor{cadmiumgreen}{green}) and a state $u=w$ with non-empty coincidence set (\textcolor{blue}{blue}).}\label{F1}
 \end{figure}
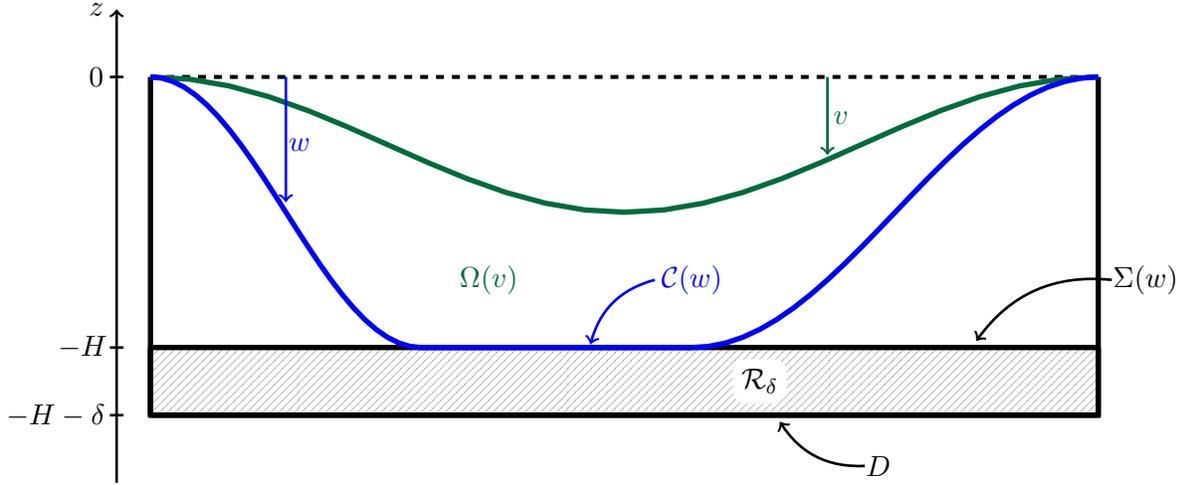

In the model considered in \cite[Section~5]{LW19}, the deflection $u$ of the elastic plate is governed by an evolution equation involving contributions from mechanical and electrostatic forces, the latter depending on the electrostatic potential denoted by $\psi$ in the following. However, for the derivation of the limiting problem for the electrostatic potential, the evolution of $u$ does not play any role. We thus
 consider throughout this paper a fixed geometry $\Omega(u)$; that is, we consider the function $u:\bar{D}\rightarrow [-H, \infty)$ with $u=0$ on $\partial D$ describing the deflection of the elastic plate as given and fixed. We refer to \cite[Section~5]{LW19} for the full model.

Given such a function $u$, the {\it electrostatic potential} $\psi=\psi_{u,\delta}$  satisfies the transmission problem
\begin{subequations}\label{TMP}
\begin{align}\label{e1}
\mathrm{div}(\sigma_\delta\nabla\psi)&=0 \quad\text{in }\ \Omega_\delta(u)\,, \\
\llbracket \psi \rrbracket = \llbracket \sigma_\delta \partial_z \psi\rrbracket &=0\quad\text{on }\  \Sigma(u)\,,\label{e2} \\
\psi &=h_{u,\delta}\quad\text{on }\ \partial\Omega_\delta(u)\,,\label{e3}
\end{align}
\end{subequations}
and the corresponding {\it electrostatic energy} of the device with geometry $\Omega_\delta(u)$ is
$$
E_{e,\delta}(u):=-\frac{1}{2}\int_{\Omega_\delta(u)}  \sigma_\delta \vert\nabla \psi_{u,\delta}\vert^2\,\rd (x,z)\,.
$$
Here, $\sigma_\delta$ is the permittivity of the device, which is different in the insulating layer and free space, and $h_{u,\delta}$ is a given suitable function describing the boundary values of the electrostatic potential. By $\llbracket \cdot \rrbracket$ we denote the jump of a function across the interface $\Sigma(u)$. The Lax-Milgram theorem provides the existence of a unique electrostatic potential $\psi_{u,\delta}=\chi_{u,\delta}+h_{u,\delta}$ solving \eqref{TMP} in a variational sense, and the function $\chi_{u,\delta}\in H_0^1(\Omega_\delta(u))$ is the minimizer of the Dirichlet integral
$$
G_\delta[\vartheta]:=\frac{1}{2}\int_{\Omega_\delta(u)}  \sigma_\delta \vert\nabla (\vartheta+h_{u,\delta})\vert^2\,\rd (x,z)
$$
among functions $\vartheta\in H_0^1(\Omega_\delta(u))$, see Proposition~\ref{P1} below.

In the following we shall derive the limiting model obtained from \eqref{TMP} as $\delta\rightarrow 0$ when imposing suitable assumptions on the function $h_{u,\delta}$ defining the boundary values of the potential (see \eqref{bobbybrown} below) and on the permittivity $\sigma_\delta$ (see \eqref{sigma} below), so that information on the dielectric heterogeneity is inherited. As for the permittivity we assume that it is constant (normalized to 1) in $\Omega(u)$ and a reinforced limit $\sigma_\delta=O(\delta)$ in $\mathcal{R}_\delta$. We then shall follow \cite{AB86} to compute the Gamma limit with respect to the $L_2$-topology  of the family of functionals $(G_\delta)_{\delta\in (0,1)}$ as $\delta\rightarrow 0$, which turns out to be the functional 
\begin{equation*}
G[\vartheta]:= \frac{1}{2}\displaystyle\int_{ \Omega(u)}   \big\vert\nabla (\vartheta+h_u)\big\vert^2\,\rd (x,z) +\frac{1}{2}\displaystyle\int_{ D} \big(\sigma \big\vert \vartheta+h_u-\mathfrak{h}_u\big\vert^2\big)(x,-H)\,\rd x 
\end{equation*}
with $h_u$ and $\mathfrak{h}_u$ defined below in \eqref{h00} and in \eqref{h0}, respectively, see Theorem~\ref{P3}. Let us emphasize here that an utmost challenging feature of the limiting problem is that $\Omega(u)$ need not be a Lipschitz set as it may have cusps when the coincidence set $\mathcal{C}(u)$ is nonempty. Therefore, the usual trace theorem is not available and a meaningful definition of $G$ requires a suitable definition of a trace on $(D \setminus \mathcal{C}(u))\times \{-H\}$ for functions in (a subset of) $H^1(\Omega(u))$, see Lemmas~\ref{lemT1H} and~\ref{lemT2H}. Once this issue is settled,
the existence of a minimizer $\chi_u$ of $G$ in a suitable subset of $H^1(\Omega(u))$ is shown by classical arguments, see Proposition~\ref{lemt0P}. The derivation of  the corresponding Euler-Lagrange equation offers further  challenges again related to the non-smoothness of $\Omega(u)$. Indeed, a formal computation reveals that $\psi_u=\chi_u+h_u$ solves Laplace's equation on $\Omega(u)$ with a Robin boundary condition along the interface $\Sigma(u)$ and a Dirichlet  condition on the other boundary parts; that is,
\begin{subequations}\label{MBP0}
\begin{align}
\Delta\psi_u&=0 \quad\text{in }\ \Omega(u)\,, \label{MBP1}\\
\psi_u &=h_u\quad\text{on }\ \partial\Omega(u)\setminus  \Sigma(u)\,,\label{MBP2}\\
- \partial_z\psi_u +\sigma (\psi_u-\mathfrak{h}_u)&=0\quad\text{on }\    \Sigma(u)\label{MBP3}\,.
\end{align}
\end{subequations}
However, a rigorous computation relies on Gau\ss ' theorem  which requires some geometric condition on the boundary of $\Omega(u)$ and the existence of boundary traces for $\nabla\chi_u$, see \cite{Koenig}\footnote{We thank Elmar Schrohe for pointing out this reference.}. Due to the Robin boundary condition the resulting model is consistent in the sense that touching plates again do not lead to a singularity in the equations.

\begin{remark}\label{rem.1}
Of course, if the coincidence set $\mathcal{C}(u)$ is empty, then $\Omega(u)$ is a Lipschitz domain and the derivation of \eqref{MBP0} only requires that $\chi_u$ belongs to $H^2(\Omega(u))$. However, this property is not guaranteed by classical elliptic regularity theory since $\Omega(u)$ is only Lipschitz. In the special case that $D$ is a one-dimensional interval and under  appropriate choices of  $h_u$ and $\mathfrak{h}_u$, we provide a rigorous justification of \eqref{MBP0} in Theorem~\ref{R1}. 
\end{remark}

In Section~\ref{S2} we first list the precise assumptions that we impose on the permittivity $\sigma_\delta$ and on the function $h_{u,\delta}$ defining the boundary values of the electrostatic potential. Moreover, since, as pointed out above, the set $\Omega(u)$ may not be Lipschitz for deformations $u$ with non-empty coincidence set $\mathcal{C}(u)$ and thus standard trace theorems are not valid, we derive in Section~\ref{S2} also boundary trace theorems in weighted spaces for functions in $H^1(\Omega(u))$.
Section~\ref{S3} is dedicated to the computation of the Gamma limit of $(G_\delta)_{\delta\in (0,1)}$ as $\delta\rightarrow 0$, which is the main result of this paper, see Theorem~\ref{P3}. Moreover, we derive in Section~\ref{S3} the limiting equations~\eqref{MBP0}.

From now on, the function $u$ is fixed and assumed to satisfy
\begin{subequations}\label{MothersFinest}
\begin{equation}\label{MothersFinest0}
u\in H_0^1(D)\cap C(\bar{D})\quad\text{with}\quad u\ge -H\ \text{ in }\ D\,,
\end{equation}
and
\begin{equation}\label{MothersFinest1}
\text{$\Omega(u)$ satisfies the segment property}
\end{equation}
\end{subequations}
in the sense of \cite[Definition~10.23]{Leoni17}.

\section{Assumptions and Auxiliary Results}\label{S2}

In this section we state the precise assumptions imposed on the permittivity $\sigma_\delta$ and the function $h_{u,\delta}$ defining the boundary values of the electrostatic potential. We also provide some auxiliary results regarding boundary traces for functions defined on the possibly non-Lipschitz set $\Omega(u)$.

\subsection{Assumptions on  $\sigma_\delta$ and $h_{u,\delta}$}

To inherit in the limit $\delta\rightarrow 0$ the information of the permittivity from the insulating layer, we specifically assume that the permittivity  scales with the layer's thickness; that is, we assume that the permittivity of the device is given in the form
\begin{subequations}\label{sigma}
\begin{equation}\label{sigmad}
 \sigma_\delta(x,z):=  \left\{ \begin{array}{ll}
\delta\sigma (x,z)\,, & (x,z)\in \mathcal{R}_\delta \,, \\
1\,, & (x,z)\in \Omega(u)\,,
\end{array} \right.
\end{equation}
for $\delta\in (0,1)$, where $ \sigma \in C(\bar D\times [-H-1,-H])$ is a fixed function with 
\begin{equation}\label{sigmam}
 \sigma_{max}:=\max_{\bar D\times [-H-1,-H]}  \sigma \,,\qquad \sigma_{min}:=\min_{\bar D\times [-H-1,-H]}  \sigma >0\,.
\end{equation}
\end{subequations}
Regarding the boundary values of the electrostatic potential given in \eqref{e3} we fix two $C^2$-functions
\begin{subequations}\label{bobbybrown}
\begin{equation}\label{bobbybrown2a}
h_b: \bar{D}\times [-H-1,-H]\times [-H,\infty)\rightarrow \R
\end{equation}
and 
\begin{equation}\label{bobbybrown2aa}
h: \bar{D}\times [-H,\infty)\times [-H,\infty)\rightarrow \R
\end{equation}
satisfying
\begin{align}
h_b(x,-H,w)&=h(x,-H,w)\,,\quad (x,w)\in D\times [-H,\infty)\,,\label{bobbybrown2}\\
  \sigma(x,-H)\partial_z h_b(x,-H,w)& = \partial_z h(x,-H,w)\,,\quad (x,w)\in D\times [-H,\infty)\,.\label{bobbybrown3}
\end{align}
\end{subequations}
We then define
\begin{equation}\label{bobbybrown40}
h_\delta(x,z,w):=  \left\{ \begin{array}{ll}
\displaystyle{h_b\left(x,-H+\frac{z+H}{\delta},w\right)}\,, & (x,z,w)\in \bar D\times [-H-\delta,-H)\times [-H,\infty)\,, \\ \vspace{-3mm}
\\
h(x,z,w)\,, & (x,z,w)\in \bar D\times [-H,\infty)\times [-H,\infty)\,,
\end{array} \right.
\end{equation}
and observe that, by \eqref{bobbybrown}, for $(x,w)\in \bar D\times  [-H,\infty)$,
\begin{equation}\label{bobbybrown42}
\begin{split}
\lim_{z\searrow -H} h_\delta(x,z,w)&=\lim_{z\nearrow -H} h_\delta(x,z,w)\,,\\
\lim_{z\searrow -H} \sigma_\delta(x,z)\partial_z h_\delta(x,z,w)&=\lim_{z\nearrow -H}\sigma_\delta(x,z)\partial_z h_\delta(x,z,w)\,.
\end{split}
\end{equation}
In the following, we shall also use the abbreviations
\begin{equation}\label{h000}
h_{u,\delta}(x,z):=h_\delta(x,z,u(x))\,,\quad (x,z)\in \Omega_\delta(u)\,,
\end{equation}
and 
\begin{equation}\label{h00}
h_{u}(x,z):=h(x,z,u(x))\,,\quad (x,z)\in \Omega(u)\,.
\end{equation}
Then \eqref{bobbybrown42} entails
\begin{equation}\label{bobbybrown44}
\llbracket h_{u,\delta} \rrbracket = \llbracket \sigma_\delta \partial_z h_{u,\delta}\rrbracket =0\quad\text{on }\  \Sigma(u)\,.
\end{equation}
Furthermore, we set
\begin{equation}\label{h0}
\mathfrak{h}_u(x,-H):=h_b(x,-H-1,u(x))\,,\quad x\in \bar D\,.
\end{equation}

\subsection{Traces in $H^1(\Omega(u))$}

As pointed out already in the introduction, the region $\Omega(u)$ need not be Lipschitz (besides not being connected) when the elastic plate touches the insulating layer; that is, when $\mathcal{C}(u)\ne \emptyset$. That there is still a meaningful definition of boundary traces on $D\setminus\mathcal{C}(u)$ for functions in $H^1(\Omega(u))$ in this case is the content of the subsequent result. We follow \cite{MNP99},  exploiting the special geometry of $\Omega(u)$ to show that traces are well-defined in weighted spaces.

\begin{lemma}\label{lemT1H} Suppose \eqref{MothersFinest} and set $M_u:= \|H+u\|_{L_\infty(D)}$.
	\begin{itemize}
	\item [{\bf (a)}] There exists a bounded linear operator 
	\begin{equation*}
	\gamma_u\in \mathcal{L}\Big( H^1(\Omega(u)),L_2\big( D\setminus\mathcal{C}(u),(H+u)\rd x \big) \Big)
	\end{equation*} 
	such that $\gamma_u\vartheta = \vartheta(\cdot, u)$ for $\vartheta\in C^1\big( \overline{\Omega(u)} \big) $ and 
	\begin{equation}
	\int_{D\setminus\mathcal{C}(u)} |\gamma_u\vartheta|^2 (H+u)\ \mathrm{d}x \le \|\vartheta\|_{L_2(\Omega(u))}^2 + 2 M_u \|\vartheta\|_{L_2(\Omega(u))} \|\partial_z \vartheta\|_{L_2(\Omega(u))}\ . \label{t1P}
	\end{equation}
	\item [{\bf (b)}] There exists a bounded linear operator 
	\begin{equation*}
	\gamma_b\in \mathcal{L}\Big( H^1(\Omega(u)),L_2\big( D\setminus\mathcal{C}(u),(H+u)\rd x \big) \Big)
	\end{equation*} 
	such that $\gamma_b\vartheta = \vartheta(\cdot, -H)$ for $\vartheta\in C^1\big( \overline{\Omega(u)} \big) $ and 
	\begin{equation}
	\int_{D\setminus\mathcal{C}(u)} |\gamma_b\vartheta|^2 (H+u)\ \mathrm{d}x \le \|\vartheta\|_{L_2(\Omega(u))}^2 + 2 M_u \|\vartheta\|_{L_2(\Omega(u))} \|\partial_z \vartheta\|_{L_2(\Omega(u))}\ . \label{t2P}
	\end{equation}
	\end{itemize}
\end{lemma}

\begin{proof}
{\bf (a)} Let $\vartheta\in C^1\big( \overline{\Omega(u)} \big)$. For $x\not\in\mathcal{C}(u)$ and $z\in (-H,u(x))$, it follows from H\"older's inequality that
\begin{align*}
\vartheta(x,u(x))^2 & = \vartheta(x,z)^2 + 2 \int_z^{u(x)} \vartheta(x,z_*) \partial_z \vartheta(x,z_*)\ \mathrm{d}z_* \\ 
& \le \vartheta(x,z)^2 + 2 \left( \int_{-H}^{u(x)} |\vartheta(x,z_*)|^2\ \mathrm{d}z_* \right)^{1/2} \left( \int_{-H}^{u(x)} |\partial_z\vartheta(x,z_*)|^2\ \mathrm{d}z_* \right)^{1/2}\ .
\end{align*}
Hence,
\begin{align*}
& (H+u)(x)  \vartheta(x,u(x))^2\\
& \qquad \le \int_{-H}^{u(x)} \vartheta(x,z)^2\ \mathrm{d}z  \\
& \qquad\qquad + 2 (H+u)(x) \left( \int_{-H}^{u(x)} |\vartheta(x,z_*)|^2\ \mathrm{d}z_* \right)^{1/2} \left( \int_{-H}^{u(x)} |\partial_z\vartheta(x,z_*)|^2\ \mathrm{d}z_* \right)^{1/2}\,.
\end{align*}
We use once more H\"older's inequality to obtain
\begin{equation}
\int_{D\setminus\mathcal{C}(u)} (H+u)(x) \vartheta(x,u(x))^2\ \mathrm{d}x \le \|\vartheta\|_{L_2(\Omega(u))}^2 + 2 M_u \|\vartheta\|_{L_2(\Omega(u))} \|\partial_z\vartheta\|_{L_2(\Omega(u))}\ . \label{t6P}
\end{equation}
Owing to \eqref{MothersFinest1}, the space $C^1\big( \overline{\Omega(u)} \big)$  is dense in $H^1(\Omega(u))$ according to \cite[Theorem~10.29]{Leoni17} or \cite[II.Theorem~3.1]{Necas67}. We then infer from \eqref{t6P} that the mapping $\vartheta\mapsto \vartheta(\cdot,u)$ from $C^1\big( \overline{\Omega(u)} \big)$ to $L_2\big( D\setminus\mathcal{C}(u),(H+u)\rd x \big)$ extends by density  to a linear bounded operator $\gamma_u$ from $H^1(\Omega(u))$ to $L_2(D\setminus\mathcal{C}(u), (H+u)\rd x)$ and which satisfies \eqref{t6P}. 

\noindent {\bf (b)} The proof being similar to that of~(a), we omit it here.
\end{proof}

For simplicity, we use the notation 
\begin{equation*}
\vartheta(x,u) := \gamma_u\vartheta(x)\,, \quad \vartheta(x,-H) := \gamma_b\vartheta(x)\,, \qquad x \in D\setminus\mathcal{C}(u)\,, \ \vartheta \in H^1(\Omega(u))\,.
\end{equation*}\

Next,  we introduce $H_B^1(\Omega(u))$ as the closure in $H^1(\Omega(u))$ of the set
\begin{equation*}
\begin{split}
C_B^1(\overline{\Omega(u)}):=\Big\{\vartheta\in C^1(\overline{\Omega(u)})\: \ 
&\vartheta(x,u(x))=0\,,\ x\in D\\
& \text{ and }\vartheta(x,z)=0\,,\ (x,z)\in \partial D\times (-H,0] \Big\}\,.
\end{split}
\end{equation*}
Since $\vartheta(x,u(x))=\vartheta(x,-H)=0$ for $x\in \mathcal{C}(u)$ and $\vartheta\in C_B^1(\overline{\Omega(u)})$, we agree upon setting $\vartheta(x,u(x))=\vartheta(x,-H):=0$ for all $x\in \mathcal{C}(u)$ and $\vartheta\in H_B^1(\Omega(u))$ in the reminder of this paper. For functions in $H_B^1(\Omega(u))$ we  derive a Poincar\'e inequality and improve the information on the trace along $\Sigma$ from Lemma~\ref{lemT1H}:

\begin{lemma}\label{lemT2H}
 Suppose \eqref{MothersFinest} and let  $\vartheta\in H_B^1(\Omega(u))$. Then
\begin{equation}
\|\vartheta\|_{L_2(\Omega(u))} \le 2 \|H+u\|_{L_\infty(D)} \|\partial_z \vartheta\|_{L_2(\Omega(u))}\,, \label{t3P}
\end{equation}
and the trace $\vartheta\mapsto \vartheta(\cdot,-H)$ yields a bounded linear operator from $H_B^1(\Omega(u))$ to $L_2(D)$ with
\begin{equation}
\|\vartheta(\cdot,-H)\|_{L_2(D)}^2 \le 2  \|\vartheta\|_{L_2(\Omega(u))} \|\partial_z \vartheta\|_{L_2(\Omega(u))}\ . \label{t4P}
\end{equation}
\end{lemma}

\begin{proof}
Consider first $\vartheta\in C_B^1\big( \overline{\Omega(u)} \big)$. Since $\vartheta(x,u(x))=0$ for $x\in D$, it follows from H\"older's inequality that, for $x\not\in \mathcal{C}(u)$ and $z\in (-H,u(x))$,
\begin{align}
|\vartheta(x,z)|^2 & = |\vartheta(x,u(x))|^2 - 2 \int_z^{u(x)} \vartheta(x,z_*) \partial_z \vartheta(x,z_*)\ \mathrm{d}z_* \nonumber\\
& \le 2 \left( \int_{-H}^{u(x)} |\vartheta(x,z_*)|^2\ \mathrm{d}z_* \right)^{1/2} \left( \int_{-H}^{u(x)} |\partial_z\vartheta(x,z_*)|^2\ \mathrm{d}z_* \right)^{1/2} \ . \label{t5P}
\end{align}
Consequently, using again H\"older's inequality gives
\begin{align*}
& \|\vartheta\|_{L_2(\Omega(u))}^2 = \int_{D\setminus \mathcal{C}(u)} \int_{-H}^{u(x)} \vartheta(x,z)^2\ \mathrm{d}z \mathrm{d}x \\
& \qquad \le 2 \int_{D\setminus \mathcal{C}(u)} (H+u)(x) \left( \int_{-H}^{u(x)} |\vartheta(x,z_*)|^2\ \mathrm{d}z_* \right)^{1/2} \left( \int_{-H}^{u(x)} |\partial_z\vartheta(x,z_*)|^2\ \mathrm{d}z_* \right)^{1/2} \ \mathrm{d}x \\
& \qquad \le 2 \|H+u\|_{L_\infty(D)} \|\vartheta\|_{L_2(\Omega(u))} \|\partial_z \vartheta\|_{L_2(\Omega(u))}\ .
\end{align*}
Hence, 
\begin{equation*}
\|\vartheta\|_{L_2(\Omega(u))} \le 2 \|H+u\|_{L_\infty(D)} \|\partial_z \vartheta\|_{L_2(\Omega(u))}\ ,
\end{equation*}
and we complete the proof of \eqref{t3P} by a density argument.

Next, consider again $\vartheta\in C_B^1\big( \overline{\Omega(u)} \big)$ and $x\not\in \mathcal{C}(u)$. We infer from \eqref{t5P} with $z=-H$ that
\begin{equation*}
\vartheta(x,-H)^2 \le 2 \left( \int_{-H}^{u(x)} |\vartheta(x,z_*)|^2\ \mathrm{d}z_* \right)^{1/2} \left( \int_{-H}^{u(x)} |\partial_z\vartheta(x,z_*)|^2\ \mathrm{d}z_* \right)^{1/2} \ .
\end{equation*} 
Since $\vartheta(x,-H)=0$ for $x\in\mathcal{C}(u)$, we use once more H\"older's inequality to obtain
\begin{align*}
\int_D \vartheta(x,-H)^2\ \mathrm{d}x & \le 2 \int_D \left( \int_{-H}^{u(x)} |\vartheta(x,z_*)|^2\ \mathrm{d}z_* \right)^{1/2} \left( \int_{-H}^{u(x)} |\partial_z\vartheta(x,z_*)|^2\ \mathrm{d}z_* \right)^{1/2}\ \mathrm{d}x \\
& \le 2 \|\vartheta\|_{L_2(\Omega(u))} \|\partial_z \vartheta\|_{L_2(\Omega(u))}\ ,
\end{align*}
which shows \eqref{t4P} for $\vartheta\in C_B^1\big( \overline{\Omega(u)} \big)$. We again complete the proof by a density argument.
\end{proof}


\section{The Reinforced Limit}\label{S3}

As announced in the introduction we shall derive the  limiting equations of \eqref{TMP} as $\delta\rightarrow 0$ when assuming the reinforced limit \eqref{sigma} on the permittivity. For this we first compute the Gamma limit of the functionals $(G_\delta)_{\delta\in (0,1)}$ and then study the behavior of the corresponding minimizers.

\subsection{The Gamma Limit of the Electrostatic Energy}

Fix $M\ge \|u\|_{L_\infty(D)}+H$, so that
\begin{equation}\label{ufix} 
-H\le u(x)\le M-H\,,\quad x\in D\,,
\end{equation}
and set 
$$
\Omega_M:=D\times (-H-1,M)\,.
$$ 
Define for $\delta\in (0,1)$
$$
G_\delta[\vartheta]:=\left\{\begin{array}{ll} \dfrac{1}{2}\displaystyle\int_{ \Omega_\delta(u)} \sigma_\delta \vert\nabla (\vartheta+h_{u,\delta})\vert^2\,\rd (x,z)\,, & \vartheta\in H_0^1(\Omega_\delta(u))\,,\\
\infty\,, & \vartheta\in L_2(\Omega_M)\setminus H_0^1(\Omega_\delta(u))\,,
\end{array}\right.
$$
with $h_{u,\delta}$ given in \eqref{h000}.
Also, for $\vartheta\in H_B^1(\Omega(u))$, we set
\begin{equation}\label{G}
G[\vartheta]:= \frac{1}{2}\int_{ \Omega(u)}   \big\vert\nabla (\vartheta+h_u)\big\vert^2\,\rd (x,z) + \frac{1}{2}\int_{ D} \big(\sigma \big\vert \vartheta+h_u-\mathfrak{h}_u\big\vert^2\big)(x,-H)\,\rd x \,,
\end{equation}
with $h_u$ and $\mathfrak{h}_u$ defined in \eqref{h00} and \eqref{h0}, respectively, and
$$
G[\vartheta]:= \infty\,,\quad \vartheta\in L_2(\Omega_M)\setminus H_B^1(\Omega(u))\,.
$$
Then the main result of the present paper is the following convergence. 

\begin{theorem}\label{P3}
Suppose \eqref{MothersFinest} and \eqref{bobbybrown}-\eqref{bobbybrown40}. Then
$$
\Gamma-\lim_{\delta\rightarrow 0} G_\delta =G\quad\text{in }\ L_2(\Omega_M)\,.
$$
\end{theorem}

For more information on Gamma convergence we refer, e.g., to \cite{DaM93}.

\begin{proof}
{\it (i) Asymptotic weak lower semi-continuity.} Considering 
\begin{equation}\label{k1}
\vartheta_\delta\rightarrow \vartheta_0\quad \text{in}\quad L_2(\Omega_M)\,,
\end{equation} 
we shall show that
$$
G[\vartheta_0]\le\liminf_{\delta\rightarrow 0} G_\delta[\vartheta_\delta]\,.
$$
Due to the definitions of the functionals we may assume without loss of generality that $\vartheta_\delta\in H_0^1(\Omega_\delta(u))$ for $\delta\in (0,1)$ and
\begin{equation}\label{1}
\sup_{\delta\in (0,1)} G_\delta[\vartheta_\delta]<\infty\,.
\end{equation}
Therefore, by \eqref{sigma},  \eqref{bobbybrown}, and \eqref{1}, we have
\begin{equation}\label{k2}
\sup_{\delta\in (0,1)}\|\nabla \vartheta_\delta\|_{L_2(\Omega(u))}<\infty\,.
\end{equation}
Thus, invoking \eqref{k1} and \eqref{k2} we may further assume that
\begin{equation}\label{2a}
\vartheta_\delta\rightharpoonup \vartheta_0\quad\text{in }\ H^1(\Omega(u))\,.
\end{equation}
Since $\vartheta_\delta$ belongs to $ H_0^1(\Omega_\delta(u))$, which  is the closure of $C^\infty_c(\Omega_\delta(u))$ in $H^1(\Omega_\delta(u))$, and $C^\infty_c(\Omega_\delta(u))\subset C_B^1(\Omega(u))$, it readily follows from the definitions of $\Omega_\delta(u)$ and $\Omega(u)$ that $\vartheta_\delta$ belongs to $H_B^1(\Omega(u))$, the latter being a closed subspace of $H^1(\Omega(u))$. Thus \eqref{2a} implies that $\vartheta_0\in H_B^1(\Omega(u))$.  Moreover,  \eqref{t4P}, \eqref{k1}, and \eqref{k2} yield
\begin{equation}\label{3}
\vartheta_\delta(\cdot,-H)\rightarrow \vartheta_0 (\cdot,-H)\quad\text{in }\ L_2\big(D\big)\,.
\end{equation}
Next, since, for each $\varepsilon>0$, there is $\delta_\varepsilon\in (0,1)$ such that
$$
\vert\sigma(x,z)-\sigma(x,-H)\vert \le \varepsilon\,,\quad (x,z)\in \mathcal{R}_{{\delta_\varepsilon}}\,,
$$
it follows from \eqref{1}  that
\begin{equation*}
\begin{split}
\liminf_{\delta\rightarrow 0} \, &\delta\int_{\mathcal{R}_\delta }\sigma(x,z)\vert\nabla(\vartheta_\delta+h_{u,\delta})\vert^2\,\rd (x,z)\\
&= \liminf_{\delta\rightarrow 0}\, \delta\int_{\mathcal{R}_\delta }\sigma(x,-H)\vert\nabla(\vartheta_\delta+h_{u,\delta})\vert^2\,\rd (x,z)\\
&\ge \liminf_{\delta\rightarrow 0} \, \delta\int_{\mathcal{R}_\delta }\sigma(x,-H)\vert\partial_z(\vartheta_\delta+h_{u,\delta})\vert^2\,\rd (x,z)\,.
\end{split}
\end{equation*}
The property   $\vartheta_\delta(\cdot,-H-\delta)= 0$ a.e. in $D$ and H\"older's inequality yield
$$
\vert (\vartheta_\delta+h_{u,\delta})(x,-H)-h_{u,\delta}(x,-H-\delta)\vert^2\le \delta \int_{-H-\delta}^{-H}\vert\partial_z(\vartheta_\delta +h_{u,\delta})(x,z)\vert^2\,\rd z
$$
for a.e. $x\in D$ while \eqref{bobbybrown2} and \eqref{bobbybrown40} imply for $x\in D$ (recalling \eqref{h00} and \eqref{h0})
$$
h_{u,\delta}(x,-H)=h_u(x,-H)\,,\qquad h_{u,\delta}(x,-H-\delta)  = \mathfrak{h}_u(x,-H) \,.
$$
Consequently,
\begin{equation*}
\begin{split}
\liminf_{\delta\rightarrow 0} \,\delta &\int_{\mathcal{R}_\delta }\sigma(x,z)\vert\nabla(\vartheta_\delta+h_{u,\delta})\vert^2\,\rd (x,z)\\
&\ge \liminf_{\delta\rightarrow 0} \int_D\sigma(x,-H)\big\vert\vartheta_\delta(x,-H)+h_u(x,-H)-\mathfrak{h}_u(x,-H)\big\vert^2\,\rd x\\
&=\int_{D}\sigma(x,-H)\big\vert\vartheta_0(x,-H)+h_u(x,-H)-\mathfrak{h}_u(x,-H)\big\vert^2\,\rd x\,,
\end{split}
\end{equation*}
where we used \eqref{3} and $\vartheta_\delta(x,-H)=0$, $x\in \mathcal{C}(u)$ (since $\vartheta_\delta\in H_0^1(\Omega_\delta(u))$)  to derive the last equality. Since $h_{u,\delta}=h_{u}$ and $\sigma_\delta=1 $  in $\Omega(u)$ it follows from \eqref{2a} that
$$
\dfrac{1}{2}\displaystyle\int_{ \Omega(u)}   \vert\nabla (\vartheta_0+h_u)\vert^2\,\rd (x,z)
\le \liminf_{\delta\rightarrow 0} \dfrac{1}{2}\displaystyle\int_{ \Omega(u)} \sigma_\delta \vert\nabla (\vartheta_\delta+h_{u,\delta})\vert^2\,\rd (x,z)\,.
$$
Therefore, gathering the last two inequalities gives
$$
\liminf_{\delta\rightarrow 0} G_\delta[\vartheta_\delta]\ge G[\vartheta_0]\,,
$$
and thus the weak lower semi-continuity of the functionals $(G_\delta)_{\delta\in (0,1)}$ follows.\\

{\it (ii) Recovery sequence.} To prove the existence of a recovery sequence it suffices, by definition of the functionals $(G_\delta)_{\delta\in (0,1)}$, to consider $\vartheta\in H_B^1(\Omega(u))$. Let $\bar\vartheta$ denote the trivial extension of $\vartheta$ to $D\times (-H,M)$ and then its reflection to $D\times(-2H-M,M)$; that is,
$$
\bar\vartheta(x,z):= \left\{ \begin{array}{ll} 0\,, & x\in D\,,\ u(x)<z<M\,, \\[0.1cm]
\vartheta(x,z)\,, &  x\in D\,,\ -H<z\le u(x)\,,\\[0.1cm]
\vartheta(x,-2H-z)\,, &  x\in D\,,\ -2H-u(x)<z\le -H\,,\\[0.1cm]
0\,, &  x\in D\,,\ -2H-M<z\le -2H-u(x)\,.\\
\end{array} \right.
$$
Let 
$$
\tau_\delta(x):= \left\{ \begin{array}{ll} 1\,, & d(x,\partial D)>\sqrt{\delta}\,,\\ \vspace{-3mm} \\
\displaystyle\frac{d(x,\partial D)}{\sqrt{\delta}}\,, & d(x,\partial D)\le\sqrt{\delta}\,, \end{array} \right.\qquad x\in D\,,
$$
where $d(\cdot,\partial D)$ denotes the distance to $\partial D$. Since the $C^2$-regularity of the boundary of $D$  implies that $d(\cdot,\partial D)$ is $C^2$ near $\partial D$ (see \cite[Lemma 14.16]{GT01}), we have $\tau_\delta\in H^1(D)$ for $\delta$ small enough. Define now
\begin{equation*}
\begin{split}
\vartheta_\delta(x,z):=&\ \frac{z+H+\delta}{\delta}\bar\vartheta(x,z) +\frac{z+H+\delta}{\delta}\big[ h_{u,\delta}(x,-H)-h_{u,\delta}(x,-H-\delta)\big] \tau_\delta (x)\\
&-\big[ h_{u,\delta}(x,z)-h_{u,\delta}(x,-H-\delta)\big]  \tau_\delta (x)\,,\quad (x,z)\in \mathcal{R}_\delta \,,
\end{split}
\end{equation*}
and
$$
\vartheta_\delta(x,z):=\vartheta(x,z)\,,\quad (x,z)\in \Omega(u)\,.
$$
The regularities of $\vartheta $,  $\bar\vartheta$, and $\tau_\delta$ imply that $\vartheta_\delta\in H^1(\mathcal{R}_\delta )\cap H^1(\Omega(u))$ and thus, since moreover $\llbracket \vartheta_\delta\rrbracket=0$ on $\Sigma(u)$, we deduce $\vartheta_\delta\in H^1(\Omega_\delta(u))$. By construction, it follows that $\vartheta_\delta$ vanishes on $\partial \Omega_\delta(u)$, hence $\vartheta_\delta\in H_0^1(\Omega_\delta(u))$. We now claim that
\begin{equation}\label{RS}
G[\vartheta]=\lim_{\delta\rightarrow 0} G_\delta[\vartheta_\delta]\,.
\end{equation}
Indeed, for $(x,z)\in \mathcal{R}_\delta $ we note that
\begin{equation}\label{p1}
\begin{split}
\partial_z(\vartheta_\delta+h_{u,\delta})(x,z)&=\frac{1}{\delta}\bar\vartheta(x,z) + \frac{1}{\delta}\big[ h_{u,\delta}(x,-H)-h_{u,\delta}(x,-H-\delta)\big] \tau_\delta (x) \\
& \qquad +\frac{z+H+\delta}{\delta}\partial_z\bar\vartheta(x,z)+\big(1-\tau_\delta (x) \big)\partial_z h_{u,\delta}(x,z)\,,
\end{split}
\end{equation}
and then handle the terms separately.
From \eqref{bobbybrown40}  and $\sigma_\delta=\delta\sigma$ in $\mathcal{R}_\delta $ we obtain
\begin{align*}
\int_{\mathcal{R}_\delta }\sigma_\delta&(x,z) \left\vert\frac{1}{\delta}\bar\vartheta(x,z) + \frac{1}{\delta}\big[ h_{u,\delta}(x,-H)-h_{u,\delta}(x,-H-\delta)\big] \tau_\delta (x)\right\vert^2\,\rd (x,z)\nonumber\\
&=
\frac{1}{\delta}\int_{-H-\delta}^{-H}\int_D\sigma(x,z) \left\vert \bar\vartheta(x,z) + \big[ h_{u}(x,-H)-\mathfrak{h}_u(x,-H)\big] \tau_\delta (x)\right\vert^2\,\rd x\rd z\,.
\end{align*}
Thus, recalling the definition of $\tau_\delta$  and using Lebesgue's theorem,
\begin{align}
\lim_{\delta\rightarrow 0}\, & \int_{\mathcal{R}_\delta }\sigma_\delta(x,z) \left\vert\frac{1}{\delta}\bar\vartheta(x,z) + \frac{1}{\delta}\big[ h_{u,\delta}(x,-H)-h_{u,\delta}(x,-H-\delta)\big] \tau_{{\delta}}(x)\right\vert^2\,\rd (x,z)\nonumber\\
&=
\int_D\sigma(x,-H) \left\vert \bar\vartheta(x,-H) +  h_{u}(x,-H)-\mathfrak{h}_u(x,-H) \right\vert^2\,\rd x\,.\label{p2}
\end{align}
Next, we have
\begin{align*}
\int_{\mathcal{R}_\delta }&\sigma_\delta(x,z) \left\vert\frac{z+H+\delta}{\delta}\partial_z\bar\vartheta(x,z)\right\vert^2\,\rd (x,z) \le
\delta \sigma_{max} \int_{-H-\delta}^{-H}\int_D \left\vert \partial_z\bar\vartheta(x,z)\right\vert^2\,\rd x\rd z\,,
\end{align*}
so that
\begin{align}\label{p3}
\lim_{\delta\rightarrow 0}\, \int_{\mathcal{R}_\delta }\sigma_\delta(x,z) \left\vert\frac{z+H+\delta}{\delta}\partial_z\bar\vartheta(x,z)\right\vert^2\,\rd (x,z)=0
\end{align}
since $\bar\vartheta \in H^1(D\times(-2H-M,M))$. 
Moreover, from \eqref{bobbybrown40} it follows that
\begin{equation}\label{pz}
\partial_z h_{u,\delta}(x,z)=\frac{1}{\delta}\partial_z h_{1}\left(x,-H+\frac{z+H}{\delta},u(x)\right)\,,\quad (x,z)\in \mathcal{R}_\delta \,,
\end{equation}
from which we get, using substitution, 
\begin{align*}
\int_{\mathcal{R}_\delta }&\sigma_\delta(x,z) \left\vert\big(1-\tau_\delta (x) \big)\partial_z h_{u,\delta}(x,z)\right\vert^2\,\rd (x,z)\nonumber\\
&\le \sigma_{max}\int_{-H-1}^{-H}\int_D \big\vert \big(1-\tau_\delta (x) \big) \partial_z h_{1}\left(x,\xi,u(x)\right)\big\vert^2\,\rd x\rd \xi\,.
\end{align*}
Hence, by definition of $\tau_\delta$ and Lebesgue's theorem,
\begin{align}\label{p4}
\lim_{\delta\rightarrow 0}\, \int_{\mathcal{R}_\delta }\sigma_\delta(x,z) \left\vert\big(1-\tau_{\delta} (x) \big)\partial_z h_{u,\delta}(x,z)\right\vert^2\,\rd (x,z)=0\,.
\end{align}
Gathering \eqref{p1}, \eqref{p2}, \eqref{p3}, and \eqref{p4} we derive
\begin{equation}\label{p7}
\begin{split}
\lim_{\delta\rightarrow 0}\, &\int_{\mathcal{R}_\delta }\sigma_\delta(x,z) \vert\partial_z(\vartheta_\delta+h_{u,\delta})\vert^2\,\rd (x,z)\\
&= \int_{D}\sigma(x,-H) \left\vert \vartheta(x,-H) +  h_{u}(x,-H)-\mathfrak{h}_u(x,-H) \right\vert^2\,\rd x \,.
\end{split}
\end{equation}
Next, still for $(x,z)\in \mathcal{R}_\delta $, we compute
\begin{equation*}
\begin{split}
\nabla_x\vartheta_\delta(x,z)=&\ \frac{z+H+\delta}{\delta}\nabla_x\bar\vartheta(x,z)\\
& +\frac{z+H+\delta}{\delta}\big[ \nabla_x h_{u,\delta}(x,-H)- \nabla_x h_{u,\delta}(x,-H-\delta)\big] \tau_\delta (x)\\
&+\frac{z+H+\delta}{\delta}\big[ h_{u,\delta}(x,-H)-  h_{u,\delta}(x,-H-\delta)\big]  \nabla_x\tau_\delta (x)\\
&-\big[ \nabla_xh_{u,\delta}(x,z)-\nabla_xh_{u,\delta}(x,-H-\delta)\big]  \tau_\delta (x)\\
&-\big[ h_{u,\delta}(x,z)-h_{u,\delta}(x,-H-\delta)\big]  \nabla_x\tau_\delta (x)\,,
\end{split}
\end{equation*}
where
\begin{equation}\label{haha}
\begin{split}
 \nabla_x h_{u,\delta}(x,z)= &\ \nabla_x h_{1} \left(x,-H+\frac{z+H}{\delta}, u(x)\right)\\
&+\partial_wh_b\left(x,-H+\frac{z+H}{\delta}, u(x)\right)\nabla_x u(x)\,.
\end{split}
\end{equation}
We further note that

$$
0\le \tau_\delta(x)\le 1\,,\qquad 0\le \frac{z+H+\delta}{\delta}\le 1
$$
for $(x,z)\in\mathcal{R}_\delta $. Gathering these observations, recalling that $\sigma_\delta=\delta\sigma$ in $\mathcal{R}_\delta$, and denoting the norm of $h_b$ in $C^1(\bar D\times[-H-1,-H]\times [-H,M])$ by $\| h_b\|_{C^1}$ we deduce
\begin{equation*}
\begin{split}
\int_{\mathcal{R}_\delta }&\sigma_\delta(x,z)\big\vert\nabla_x(\vartheta_\delta+h_{u,\delta})\big\vert^2\rd (x,z)\\
&\le \  c\, \delta \sigma_{max}\int_{\mathcal{R}_\delta } \vert\nabla_x\bar\vartheta(x,z)\vert^2 \, \rd (x,z)\\
&\qquad +c\, \delta \sigma_{max}  \| h_b\|_{C^1}^2 \int_{-H-\delta}^{-H}\int_D \left(1+\vert\nabla_x u(x)\vert^2+\vert \nabla_x\tau_\delta(x)\vert^2\right)\,\rd x\rd z\\
&\le \  c\, \delta \sigma_{max}\int_{\mathcal{R}_\delta } \vert\nabla_x\bar\vartheta(x,z)\vert^2\, \rd (x,z)\\
&\qquad + c\, \delta^2 \sigma_{max} \| h_b\|_{C^1}^2 \int_D   \left(1+\vert\nabla_x u(x)\vert^2+\vert \nabla_x\tau_\delta(x)\vert^2\right) \,\rd x\,.
\end{split}
\end{equation*}
Now, since the distance function $d(\cdot,\partial D)\in C^2$ (see \cite[Lemma 14.16]{GT01}) satisfies the eikonal equation we have
$$
 \vert\nabla_x\tau_\delta(x)\vert\le \frac{1}{\sqrt{\delta}}\,,\quad x\in D\,,
$$
and since $u\in H_0^1(D)$ and $\bar\vartheta \in H^1(D\times(-2H-M,M))$, we deduce that the right-hand side of the above estimate is of order~$\delta$, hence
\begin{equation}\label{p9}
\lim_{\delta\rightarrow 0}\, \int_{\mathcal{R}_\delta }\sigma_\delta(x,z)\big\vert\nabla_x(\vartheta_\delta+h_{u,\delta})\big\vert^2\,\rd (x,z)=0\,.
\end{equation}
Consequently, we derive from \eqref{sigmad}, \eqref{bobbybrown40}, \eqref{p7}, \eqref{p9}, and \eqref{bobbybrown2}
\begin{equation*}
\begin{split}
\lim_{\delta\rightarrow 0} \displaystyle&\int_{ \Omega_\delta(u)} \sigma_\delta \vert\nabla (\vartheta_\delta+h_{u,\delta})\vert^2\,\rd (x,z)\\
=&\ \int_{\Omega(u)}  \big\vert\nabla(\vartheta+h_u)\big\vert^2\,\rd (x,z)\\
&\ +
\int_{D}\sigma(x,-H) \left\vert \vartheta(x,-H) +  h_{u}(x,-H)-\mathfrak{h}_{u}(x,-H) \right\vert^2\,\rd x \\
= &\ 2G[\vartheta];
\end{split}
\end{equation*}
that is, \eqref{RS} since $\vartheta_\delta\in H_0^1(\Omega_\delta(u))$. This proves the assertion.
\end{proof}

\subsection{Minimizers}

Now that we have shown the Gamma convergence of the functionals $(G_\delta)_{\delta\in (0,1)}$ towards $G$, we can deduce useful information on the relation between their minimizers. 
We first recall from the Lax-Milgram theorem (also see \cite[Proposition~3.1, Lemma~3.2]{LW19}) the following result regarding the solvability of the transmission problem~\eqref{TMP}:

\begin{proposition}\label{P1}
Suppose \eqref{MothersFinest} and \eqref{bobbybrown}-\eqref{bobbybrown40}.
For  $\delta\in (0,1)$, there is a unique 
minimizer $\chi_{u,\delta} \in  H_{0}^1(\Omega_\delta(u))$ of the functional $G_\delta$ on $ H_{0}^1(\Omega_\delta(u))$. It satisfies
\begin{equation}\label{est}
\int_{\Omega_\delta(u)} \sigma_\delta\vert\nabla\chi_{u,\delta}\vert^2\,\rd(x,z)\le 4 \int_{\Omega_\delta(u)} \sigma_\delta\vert\nabla h_{u,\delta}\vert^2\,\rd(x,z)\,.
\end{equation}
In addition, $\psi_{u,\delta}:=\chi_{u,\delta}+h_{u,\delta}$ is a variational solution to \eqref{TMP}.
\end{proposition}

As for the functional $G$ we note:

\begin{proposition}\label{lemt0P}
Suppose \eqref{MothersFinest} and \eqref{bobbybrown}. There is a unique minimizer  $\chi_u\in H_B^1(\Omega(u))$ of the functional $G$ on $H_B^1(\Omega(u))$.  It satisfies
\begin{equation*}
	\|\nabla\chi_u\|_{L_2(\Omega(u))}^2 + \|\sqrt{\sigma} \chi_u(\cdot,-H)\|_{L_2(D)}^2 \le 4\|\nabla h_u\|_{L_2(\Omega(u))}^2+4\|\sqrt{\sigma} (h_u-\mathfrak{h}_u)(\cdot,-H)\|_{L_2(D)}^2 \, .
	\end{equation*}
\end{proposition}

\begin{proof}
It readily follows from  \eqref{sigmam}, the Poincar\'e inequality \eqref{t3P},  and the Lax-Milgram theorem that there is a unique minimizer $\chi_u\in H_B^1(\Omega(u))$ of $G$ on $H_B^1(\Omega(u))$. Since $\chi_u$ satisfies
	\begin{equation*}
	G[\chi_u] \le G[\vartheta]\,, \qquad \vartheta\in  H_B^1(\Omega(u))\,,
	\end{equation*}
we obtain the claimed estimate by taking $\vartheta\equiv 0$ in the previous inequality.
\end{proof}

As a consequence of Theorem~\ref{P3} and Propositions~\ref{P1} and~\ref{lemt0P}, we obtain the convergence of the minimizers of the functionals.

\begin{corollary}\label{C3}
Suppose \eqref{MothersFinest} and \eqref{bobbybrown}-\eqref{bobbybrown40}. Then
$$
\chi_{u,\delta} \longrightarrow \chi_u \quad\text{ in }\ L_2(\Omega(u)) \qquad \text{and}\qquad \chi_{u,\delta} \rightharpoonup \chi_u \quad\text{ in }\ H^1(\Omega(u))
$$
as $\delta\rightarrow 0$, and
$$
\lim_{\delta\rightarrow 0} \, G_\delta\left[\chi_{u,\delta} \right]=G[\chi_u]\,.
$$
\end{corollary}

\begin{proof}
Let $\delta\in (0,1)$. We  use \eqref{sigmad}, \eqref{pz}, and \eqref{haha}  to obtain that
$$
\int_{\Omega_\delta(u)}\sigma_\delta\vert \nabla h_{u,\delta}\vert^2\,\rd (x,z)\le c_1
$$
for a constant $c_1>0$  independent of $\delta$. This estimate, along with \eqref{sigmam} and \eqref{est}, implies that
\begin{equation}\label{t3}
\|\nabla\chi_{u,\delta}\|_{L_2(\Omega(u))}^2+\delta \|\nabla\chi_{u,\delta}\|_{L_2(\mathcal{R}_\delta )}^2\le c_1\,.
\end{equation}
On the one hand, we infer from the Poincar\'e inequality \eqref{t3P} and  \eqref{t3} that
\begin{equation}\label{t4a}
\|\chi_{u,\delta}\|_{H^1(\Omega(u))}\le c_2\,.
\end{equation}
On the other hand, since $\chi_{u,\delta}(\cdot,-H-\delta)\equiv 0$ we can use the same argument as for the derivation of \eqref{t3P} to show that
\begin{equation*}
\|\chi_{u,\delta}\|_{L_2(\mathcal{R}_\delta )}\le 2\delta \|\partial_z \chi_{u,\delta}\|_{L_2(\mathcal{R}_\delta )}\,.
\end{equation*}
Therefore, using \eqref{t3} and the trivial extension of $\chi_{u,\delta}$,
\begin{equation}\label{t4}
\|\chi_{u,\delta}\|_{L_2(D\times (-H-1,-H))}\le c_3\sqrt{\delta}\,.
\end{equation}
Now, despite of the possible non-Lipschitz character of $\Omega(u)$, the
 embedding of $H^1(\Omega(u))$ in $L_2(\Omega(u))$ is compact, see \cite[I.Theorem~1.4]{Necas67} or \cite[Theorem~11.21]{Leoni17},  and we infer from \eqref{t4a} and \eqref{t4} that there are a function $\zeta\in L_2(\Omega_M)\cap H^1(\Omega(u))$  vanishing in $\Omega_M\setminus\Omega(u)$ and a sequence $\delta_n\rightarrow 0$  such that $\chi_{u,\delta_n}\rightarrow \zeta $ in $L_2(\Omega_M)$ and $\chi_{u,\delta_n}\rightharpoonup \zeta$ in $H^1(\Omega(u))$. Thus, Theorem~\ref{P3} and the fundamental theorem of $\Gamma$-convergence \cite[Corollary~7.20]{DaM93}  imply that $\zeta$ is a minimizer of $G$ on $L_2(\Omega_M)$ and that 
$$
\lim_{{n}\rightarrow \infty} G_{\delta_n}\left[ \chi_{u,\delta_n}\right]=G[\zeta]\,.
$$
Obviously, this implies that $\zeta\vert_{\Omega(u)}\in H_B^1(\Omega(u))$ is a minimizer of $G$ on $H_B^1(\Omega(u))$, hence $\zeta\vert_{\Omega(u)}=\chi_u$ by Proposition~\ref{lemt0P} and it is  independent of the sequence $(\delta_n)_{n\in\N}$.  Clearly, $G[\zeta]$ only depends on $\zeta\vert_{\Omega(u)}=\chi_u$ and is independent of the sequence $(\delta_n)_{n\in\N}$.
 This proves the claim.
\end{proof}

\subsection{The Limiting Model}

Finally, we shall derive the analogue to equations~\eqref{TMP} satisfied by $\psi_u:=\chi_u+h_u$ for which we {\it suppose} that $\chi_u\in H_B^1(\Omega(u))\cap H^2(\Omega(u))$ and that Gau\ss ' theorem applies for $\Omega(u)$ and $\nabla\chi_u$ (which requires some geometric condition on the boundary of $\Omega(u)$ and the existence of boundary traces for $\nabla\chi_u$, see \cite{Koenig}).

Let $u\in H_0^1(D)\cap C(\bar{D})$ satisfy \eqref{ufix}. Since $\chi_u\in H_B^1(\Omega(u))$ is the minimizer of $G$ on $H_B^1(\Omega(u))$ by Corollary~\ref{C3}, it satisfies the variational equality
\begin{equation*}
\begin{split}
0= &\int_{\Omega(u)}   \nabla (\chi_u+h_u)\cdot \nabla \phi\,\rd (x,z)\\
& +\int_{D} \sigma(x,-H) \big(\chi_u(x,-H)+h_u(x,-H)-\mathfrak{h}_u(x,-H)\big) \phi(x,-H)\,\rd x
\end{split}
\end{equation*}
for any $\phi\in H_B^1(\Omega(u))$. Then, by standard computations,  
\begin{equation*}
\begin{split}
0= &\int_{\Omega(u)}   \nabla (\chi_u+h_u)\cdot \nabla \phi\,\rd (x,z)\\
& +\int_{D} \sigma(x,-H) \big(\chi_u(x,-H)+h_u(x,-H)-\mathfrak{h}_{u}(x,-H)\big) \phi(x,-H)\,\rd x\\
= &-\int_{\Omega(u)}   \Delta (\chi_u+h_u)  \phi\,\rd (x,z)+\int_{\partial\Omega(u)}   \nabla (\chi_u+h_u) \cdot {\bf n}_{\partial\Omega(u)} \phi\,\rd S\\
& +\int_{D} \sigma(x,-H) \big(\chi_u(x,-H)+h_u(x,-H)-\mathfrak{h}_{u}(x,-H)\big) \phi(x,-H)\,\rd x\,.
\end{split}
\end{equation*}
Therefore, since $\phi$ vanishes on $\partial\Omega(u)\setminus \Sigma$ ,
\begin{equation}
\begin{split}
0 = &-\int_{\Omega(u)}   \Delta (\chi_u+h_u)  \phi\,\rd (x,z)-\int_{D}   \partial_z (\chi_u+h_u)(\cdot,-H) \phi(\cdot,-H)\,\rd x\\
& +\int_{D} \sigma(\cdot,-H) \big(\chi_u(\cdot,-H)+h_u(\cdot,-H)-\mathfrak{h}_{u}(\cdot,-H)\big) \phi(\cdot,-H)\,\rd x\,.
\end{split} \label{Metallica}
\end{equation}
Consequently, in this case $\psi_u=\chi_u+h_u\in H^2(\Omega(u))$ solves Laplace's equation with mixed boundary conditions of Dirichlet and Robin type as announced in \eqref{MBP0}. The corresponding electrostatic energy is
$$
E_e(u):=-G[\psi_u-h_u]\,;
$$
that is,
\begin{align*}
E_e(u)= &-\dfrac{1}{2}\displaystyle\int_{\Omega(u)}   \big\vert\nabla \psi_u\big\vert^2\,\rd (x,z)\\
& -\dfrac{1}{2}\displaystyle\int_{ D} \sigma(x,-H) \big\vert \psi_u(x,-H)-\mathfrak{h}_u(x,-H)\big\vert^2\,\rd x \,,
\end{align*}
 where $\mathfrak{h}_u$ is defined in \eqref{h0}. By Corollary~\ref{C3} we have
$E_{e,\delta} (u) \rightarrow E_e(u)$ as~$\delta\rightarrow 0$.\\

For the special case that $D$ is an interval in $\R$ and $h_b$ is of the form
\begin{equation}\label{MF}
h_b(x,z,w)=h(x,-H,w)+(z+H)\big(h(x,-H,w)-\mathfrak{h}(x,w)\big)
\end{equation}
for $(x,z,w)\in \bar D\times [-H-1,-H]\times [-H,\infty)$ with $\mathfrak{h}\in C^2(\bar D\times [-H,\infty))$, the previous computation can be rigorously justified.

\begin{theorem}\label{R1}
If $D=(a,b)\subset\R$, $u\in H^2(D)\cap H_0^1(D)$ with $u\ge -H$ in $D$, and $h$ and $h_b$ satisfy \eqref{MF}, then \eqref{MBP0} admits a unique solution $\psi_u \in H^2(\Omega(u))$ with $\partial_z \psi_u(\cdot,-H)\in L_2(D\setminus \mathcal{C}(u))$. It is given as $\psi_u=\chi_u+h_u$ with $\chi_u\in H_B^1(\Omega(u))$ being the unique minimizer of $G$ on $H_B^1(\Omega(u))$. 
\end{theorem}

\begin{remark}
We point out once more that, since $\Omega(u)$ need not be Lipschitz, the stated $L_2$-regularity of $\partial_z \psi_u(\cdot,-H)$ does not follow from the $H^2$-regularity of $\psi_u$ by a standard trace theorem. In fact, Lemma~\ref{lemT1H} only ensures that $\partial_z \psi_u(\cdot,-H)$ belongs to the weighted space $L_2(D\setminus \mathcal{C}(u),(H+u)\rd x)$. That $\partial_z \psi_u(\cdot,-H)$ additionally belongs to  $L_2(D\setminus \mathcal{C}(u))$ follows \textit{a posteriori} from \eqref{MBP3}, as shown in the proof below.
\end{remark}

\begin{proof}[Proof of Theorem~\ref{R1}]
Since $D$ is a one-dimensional interval, $H^2(D)$ is embedded in $C(\bar D)$, which implies \eqref{MothersFinest0} as well as \eqref{MothersFinest1} by \cite[Exercise~10.26]{Leoni17}. It follows from \cite{BR91, LNW} that the unique minimizer $\chi_u\in H_B^1(\Omega(u))$ of $G$ on $H_B^1(\Omega(u))$ belongs to $H^2(\Omega(u))$. Thanks to \eqref{MothersFinest1} and \cite[Theorem~10.29]{Leoni17}, see also \cite[II.Theorem~3.1]{Necas67}, there is a sequence $(\chi_{u,j})_{j\ge 1}$ in $C^\infty\big(\overline{\Omega(u)}\big)$ such that
\begin{equation}
\lim_{j\to\infty} \|\chi_{u,j} -\chi_u\|_{H^2(\Omega(u))} = 0 \,. \label{MotleyCrue} 
\end{equation}

Now, for $j\ge 1$ and $\phi\in C^1_B\big(\overline{\Omega(u)}\big)$, we infer from \cite[Folgerung~7.5]{Koenig} and the regularity of $\chi_{u,j}$, $\phi$, and $h$ that Gau\ss' theorem can be applied in each connected component of $\Omega(u)$, as there are at most two singular points. Therefore, we obtain
\begin{equation*}
\begin{split}
& \int_{\Omega(u)}   \nabla (\chi_{u,j}+h_u)\cdot \nabla \phi\,\rd (x,z)\\
& \qquad = -\int_{\Omega(u)}   \Delta (\chi_{u,j}+h_u)  \phi\,\rd (x,z)+\int_{\partial\Omega(u)}   \nabla (\chi_{u,j}+h_u) \cdot {\bf n}_{\partial\Omega(u)} \phi\,\rd S\\
& \qquad = -\int_{\Omega(u)} \Delta (\chi_{u,j}+h_u)  \phi\,\rd (x,z) - \int_{D}  \partial_z (\chi_{u,j}+h_u)(\cdot,-H) \phi(\cdot,-H)\,\rd x\,,
\end{split}
\end{equation*}
since $\phi$ vanishes on $\partial\Omega(u)\setminus \Sigma(u)$. Now, thanks to \eqref{MotleyCrue}, it is straightforward to pass to the limit $j\to\infty$ in the two integrals over $\Omega(u)$. Moreover, by Lemma~\ref{lemT1H} and \eqref{MotleyCrue}, 
\begin{equation*}
\lim_{j\to \infty} \int_{D\setminus\mathcal{C}(u)} \left| \partial_z \chi_{u,j}(\cdot,-H) - \partial_z \chi_u(\cdot,-H) \right|^2 (H+u)\,\rd x = 0.
\end{equation*}
Consequently, if $\phi(\cdot,-H)$ is compactly supported in $D\setminus \mathcal{C}(u)$, then
\begin{equation*}
\lim_{j\to\infty} \int_{D}  \partial_z (\chi_{u,j}+h_u)(\cdot,-H) \phi(\cdot,-H)\,\rd x = \int_D \partial_z (\chi_u+h_u)(\cdot,-H) \phi(\cdot,-H)\,\rd x\,.
\end{equation*}
Thanks to the above analysis, the identity \eqref{Metallica} holds true for any test function $\phi\in C^1_B\big(\overline{\Omega(u)}\big)$ such that $\phi(\cdot,-H)$ is compactly supported in $D\setminus \mathcal{C}(u)$. We then deduce from \eqref{Metallica} that $\psi_u=\chi_u+h_u$ satisfies \eqref{MBP1} and \eqref{MBP2} in $L_2(\Omega(u))$ and $L_2(\partial\Omega(u)\setminus \Sigma(u))$, respectively, while 
\begin{equation}
\partial_z\psi_u(\cdot,-H) = \big( \sigma (\chi_u + h_u - \mathfrak{h}_u) \big)(\cdot,-H) \ \text{ a.e. in }\ D\setminus\mathcal{C}(u)\,. \label{U2}
\end{equation}
Since $\chi_u\in H_B^1(\Omega(u))$, the right-hand side of \eqref{U2} belongs to $L_2(D\setminus\mathcal{C}(u))$ by Lemma~\ref{lemT2H} and the regularity of $h_u$ and $\mathfrak{h}_u$, so that $\partial_z\psi_u(\cdot,-H)$ also belongs to that space.
\end{proof}

The analysis of the complete MEMS model coupling \eqref{MBP0} to an equation for $u$ is performed in a forthcoming research~\cite{LNW} for $n=1$.

\bibliographystyle{siam}
\bibliography{BGM}

\end{document}